\documentclass[a4paper,12pt,english,reqno]{amsart}
\usepackage[utf8]{inputenc}
\usepackage[main=english]{babel}
\usepackage{amsmath,amssymb,amsthm}
\usepackage{fullpage}
\usepackage{url}
\usepackage{hyperref}

\newtheorem{thm}{Theorem}

\newtheorem{lem}[thm]{Lemma}

\newtheorem{conj}[thm]{Conjecture}
\theoremstyle{definition}

\theoremstyle{remark}
\newtheorem{rem}[thm]{Remark}

\newcommand{\R}{{\mathbb R}}
     \title{A proof of Hare's $\epsilon$ unfair conjecture}

     \author[1]{Tho Nguyen Xuan}
     \address{Faculty of Mathematics and Informatics, Hanoi University of Science and Technology \\
Hanoi, Vietnam}
     \email{tho.nguyenxuan1@hust.edu.vn}
 \author[2]{Duc Hiep Pham}
 \address{University of Education, Vietnam National University, 144 Xuan Thuy, Cau Giay, Hanoi, Vietnam}
     \email{phamduchiep@vnu.edu.vn}
        \author{Huong Pham Lan}
     \address{Department of Mathematics, Hanoi Pedagogical University 2, Nguyen Van Linh Road, Xuan Hoa, Phuc Yen, \\
Phu Tho, Vietnam}
     \email{phamlanhuong@hpu2.edu.vn}
     \date{}
 \begin{document}
     \begin{abstract} In this paper, we  prove Hare's $\epsilon$-unfair conjecture in
[Computational progress on the unfair 0--1 polynomial conjecture, 
Experiment. Math. (2025), \url{https://doi.org/10.1080/10586458.2025.2527786}].
     \end{abstract}
     \maketitle
     \section{Introduction}
   A 0--1 polynomial is a polynomial with only coefficients in $\{0,1\}$. The unfair 0--1 polynomial conjecture states the following:
   \begin{conj}\label{C1} If a 0--1 polynomial $c(x)$ is written as $c(x)=a(x)b(x)$ with $a(x),b(x)$ monic polynomials with nonnegative real coefficients, then all coefficients of $a(x)$ and $b(x)$ are $0$ or $1$.
   \end{conj}
   An equivalent form of Conjecture \ref{C1} is the following.
   \begin{conj}\label{C2}
Let $X$ and $Y$ be independent random variables taking finitely many values in $\{0,1,2,\ldots\}$. Suppose that $X+Y$ is uniform on the values it takes. Then $X$ and $Y$ are uniform as well.
   \end{conj}
According to Ghidelli \cite{Ghidelli} and Hare \cite{Hare},   Conjecture \ref{C2} was first stated by Letac in 1969. Conjecture \ref{C2} is also Problem 28  in the list of 100 open problems by Ben Green \cite{Green}. For the history  of Conjectures \ref{C1} and \ref{C2}, see \cite{Ghidelli} and the references therein. Conjecture \ref{C1} is still open and some partial results are known:
\begin{itemize}
\item 
Dvorsky \cite{Dvorsky} verified Conjecture \ref{C1} for 0--1 polynomials divisible by a polynomial of form $1+tx^{2}+x^{k}$ for all odd integers
$k\ge 341$ and real numbers $t\in (0,1)$;
\item Ghidelli \cite{Ghidelli} verified Conjecture \ref{C1} for 0--1 polynomials divisible by a polynomial of form $1+tx^{2}+x^{5}$ for all real numbers $t\in (0,1)$;
\item  Hare \cite{Hare} verified Conjecture \ref{C1} for almost all polynomials of degree up to $15$;
\item Zhang \cite{Zhang} verified Conjecture \ref{C1}  for all polynomials of degree up to $45$. \end{itemize}
In \cite{Hare}, Hare formulated the following conjecture.
\begin{conj}\cite[Conjecture~6.4]{Hare}\label{C3} For all $\epsilon>0$, there exists an $\epsilon$-unfair polynomial.
   \end{conj}
   Recall that, see \cite[Definition~5.1]{Hare}, a 0--1 polynomial $c(x)$ is called $\epsilon$-unfair if there exists a factorization $c(x) = a(x) b(x)$ with $a(x) = \sum a_i x^i$, $b(x) = \sum b_i x^i$ and 
    \begin{itemize}
        \item $b_0 = a_0 = 1$,
        \item There exists an $i$ such that $b_i \not\in \{0,1\}$,
        \item $-\epsilon \leq b_i \leq 1+\epsilon$,
        \item $-\epsilon \leq a_i \leq 1+\epsilon$.
    \end{itemize}
The goal of this paper is to prove Conjecture \ref{C3}. The main result of this paper is the following theorem.
\begin{thm}\label{T1}
Let $n$ be a positive integer. Consider the 0--1 polynomial 
\[C_n(x)=1+x^3+x^{6n+2}+x^{6n+3}+x^{6n+4}+x^{6n+5}+x^{6n+6}+x^{6n+7}+x^{12n+6}+x^{12n+9}.\]
Let $m=6n+3$. Then there exists $t_n\in (0,\dfrac{3.1}{m^2})$ such that \[A_n(x):=x^3+t_nx^2+t_nx+1\]
divides $C_n(x)$ in $\R[x]$. The quotient $B_n:=C_n/A_n$ is a real polynomial of degree $12n+6$ with $B_n(0)=1$, and the coefficient of $x$ in $B_n$ is $-t_n$. Let \[\alpha_n:=\min\{\alpha\geq 0\colon \text{ every coefficient of $A_n$ and $B_n$ lies in $[-\alpha,1+\alpha]$}\}.\]
Then \begin{equation}\label{E1}t_n\leq \alpha_n\leq \left(1+\dfrac{24}{m}\right)t_n,\quad |t_n-\dfrac{3}{m^2}|\leq \dfrac{7}{m^4}.\end{equation}
In particular, \begin{equation}\label{alpha_n}\alpha_n=\dfrac{3}{m^2}+O\left(\dfrac{1}{m^3}\right)=\dfrac{3}{(6n+3)^2}\left(1+O\left(\dfrac{1}{n}\right)\right).\end{equation}
\end{thm}
The consequence of Theorem \ref{T1} is the following. 
\begin{thm}\label{T2} For every $\epsilon>0$, there exist a 0--1 polynomial $c(x)$ and monic real polynomials $a(x),b(x)$ with $c(x)=a(x)b(x)$, $a(0)=b(0)=1$, every coefficient of $a(x)$ and $b(x)$ in $[-\epsilon,1+\epsilon]$, and some coefficient of $b(x)$ not in $\{0,1\}$. We can take $c(x)=C_n(x)$ for any $n$ such that $12(6n+3)^{-2}\leq \epsilon$.  In particular, Conjecture \ref{C3} is true.  Moreover, with $\alpha_n$ as in Theorem \ref{T1}, we have
\begin{equation}\label{E2}\lim\limits_{n\rightarrow \infty}\dfrac{\log \alpha_n}{\log(12n+9)}=-2.\end{equation}
\end{thm}
In \cite{Hare}, Hare observed that \[\log \alpha_n\approx -2.925903281-1.871057363 \log(12n + 9)\] from his data for $1\leq n\leq 40$. Theorem \ref{T2} shows that the true value of the limit is $-2$. 
\section{Proof of Theorem \ref{T1}}
Note that $C_n(x)$ has a factorization
\[
C_n(x)=(1+x^3)Q_n(x),\quad Q_n(x):=1+x^{6n+2}+x^{6n+3}+x^{6n+4}+x^{12n+6}.
\]
Hence, \[E:=\{0,m-1,m,m+1,2m\}\]
 is the exponent set of $Q_n(x)$. 
For $0<t<1$, then \[A_t(x):=x^3+tx^2+tx+1=(1+x)q_t(x),\quad q_t(x):=x^2+(t-1)x+1.\]
In order for $A_t(x)|C_n(x)$, every root of $q_t(x)$ be must be a root of $q_t(x)$. 
The roots of $q_t(x)$ are $e^{\pm i\theta}$, where $\theta\in (\pi/3,\,\pi/2)$ with $2\cos \theta=1-t$. For $x=e^{i\theta}$ then \begin{equation}\label{E5}
e^{-im\theta}Q_n(e^{i\theta})=e^{-im\theta}+e^{-i\theta}+1+e^{i\theta}+e^{im\theta}=G(\theta),
\end{equation}
where \[
G(\theta)=2\cos(m\theta)+2\cos \theta+1.
\]
Since $m\pi/3=(2n+1)\pi$, we have for all real $\delta$ then \begin{equation}\label{E6}
G(\dfrac{\pi}{3}+\delta)=1-2\cos(m\delta)+\cos\delta-\sqrt{3}\sin\delta.
\end{equation}
Through this paper, we often use the following inequalities concerning the sine and cosine functions: for all $y\in [0,1]$ then 
\begin{equation}\label{ineq}
\begin{cases}
1-\dfrac{y^2}{2}\leq\cos y\leq 1-\dfrac{y^2}{2}+\dfrac{y^4}{24},\\
y-\dfrac{y^3}{6}\leq \sin y\leq y,\\
\cos y\geq 1-\dfrac{y^2}{2}+\dfrac{y^4}{24}-\dfrac{y^6}{720}.
\end{cases}
\end{equation}

\begin{lem}\label{L1}
For every $n\geq 1$, there exists $\delta^{*}\in \left(\dfrac{\sqrt{3}}{m^2},\dfrac{\sqrt{3}}{m^2}+\dfrac{3}{m^4}\right)$
such that $G(\pi/3+\delta^*)=0$. Let $\theta^*=\pi/3+\delta^*$ and $t:=1-2\cos\theta^*$. Then \begin{equation}\label{E8}
\dfrac{3}{m^2}-\dfrac{2}{m^6}\leq t\leq \dfrac{3}{m^2}+\dfrac{7}{m^4}<\dfrac{3.1}{m^2}<1.
\end{equation}
Moreover, $\delta^*<1.78m^{-2}$ and $\theta^*\in (\pi/3,\pi/2)$.
\end{lem}
\begin{proof}
Let $\delta_1=\sqrt{3}/m^2$ and $\delta_2=\sqrt{3}/m^2+3/m^4$. Let $y_1=m\delta_1$ and $y_2=m\delta_2$. For $m\geq 9$, \[\delta_2=\dfrac{\sqrt{3}}{m^2}+\dfrac{3}{m^4}\leq \dfrac{\sqrt{3}}{m^2}+\dfrac{1}{27m^2}<\dfrac{1.78}{m^2},\]
\[y_2=m\delta_2<\dfrac{1.78}{m}.\]
From  \eqref{E6} and \eqref{ineq}, we have 
\[\begin{split}
G\left(\dfrac{\pi}{3}+\delta_1\right)&\leq \left[-1+y_1^2-\dfrac{y_1^4}{12}+\dfrac{y_1^6}{360}\right] +\left[1-\dfrac{\delta_1^2}{2}+\dfrac{\delta_1^4}{24}\right]+\left[-\sqrt{3}\delta_1+\dfrac{\sqrt{3}\delta_1^3}{6}\right]\\
&=\left(
-1+\frac{3}{m^2}-\frac{3}{4m^4}+\frac{3}{40m^6}
\right)+\left(
1-\frac{3}{2m^4}+\frac{3}{8m^8}
\right)+\left(
-\frac{3}{m^2}+\frac{3}{2m^6}
\right)\\
&=\dfrac{-90m^4+63m^2+15}{40m^8}\\
&<0,
\end{split}\]
and
\[\begin{split}
G\left(\dfrac{\pi}{3}+\delta_2\right)&\geq \left[-1+y_2^2-\dfrac{y_2^4}{12}\right]+\left[1-\dfrac{\delta_2^2}{2}\right]-\sqrt{3}\delta_2\\
&=y_2^2-\sqrt{3}\delta_2-\dfrac{y_2^4}{12}-\dfrac{\delta_2^2}{2},
\end{split}\]
Note that \[y_2^2-\sqrt{3}\delta_2=\dfrac{3\sqrt{3}}{m^4}+\dfrac{9}{m^6}>\dfrac{3\sqrt{3}}{m^4}.\]
Hence,
\[G\left(\dfrac{\pi}{3}+\delta_2\right)>\dfrac{3\sqrt{3}-\dfrac{1.78^4}{12}-\dfrac{1.78^2}{2}}{m^4}>0.\]
By continuity, there exists $\delta^*\in (\delta_1,\delta_2)$ such that $G(\pi/3+\delta^*)=0$.
Finally, \[t=1-2\cos(\dfrac{\pi}{3}+\delta^*)=1-\cos\delta^*+\sqrt{3}\sin\delta^*>0.\]
By \eqref{ineq}, we also have
\[\begin{split}t&\leq \dfrac{(\delta^*)^2}{2}+\sqrt{3}\delta^*\\
&<\dfrac{\delta_2^2}{2}+\sqrt{3}\delta_2\\
&<\dfrac{1}{2}\left(\dfrac{1.78}{m^2}\right)^2+\dfrac{3}{m^2}+\dfrac{3\sqrt{3}}{m^4}\\
&<\dfrac{3}{m^2}+\dfrac{7}{m^4},\end{split}\]
and \[\begin{split}
t&\geq \sqrt{3}\sin\delta^*\\
&\geq \sqrt{3}\left(\delta^*-\dfrac{(\delta^*)^3}{6}\right)\\
&>\sqrt{3}\left(\delta_1-\dfrac{\delta_2^3}{6}\right)\\
&>\dfrac{3}{m^2}-\dfrac{\sqrt{3}\times 1.78^3}{6m^6}\\
&>\dfrac{3}{m^2}-\dfrac{2}{m^6}.
\end{split}\]
Therefore, \eqref{E8} is proved. For the last assertions, $\delta^*<\delta_2<1.78m^{-2}$ and $t\in(0,1)$ gives $\theta^*\in (\pi/3,\pi/2)$.
\end{proof}
Now fix $\theta^*=\pi/3+\delta^*$ and $t=t_n$ as in Lemma \ref{L1}. Let $q(x):=q_t(x)=x^2-2\cos(\theta^*)x+1$ and $A_n(x):=A_{t_n}(x)$.
\begin{lem}\label{L2} In $\R[x]$, we have $A_n(x)|C_n(x)$. Let \[B_n(x):=\dfrac{C_n(x)}{A_n(x)}=\dfrac{(1-x+x^2)Q_n(x)}{q(x)}.\]
Then $B_n(x)\in \R[x]$ with $\deg B_n(x)=12n+6=2m$ and $B_n(0)=1$.
\end{lem}
\begin{proof}
It suffices to show that all roots of $q(x)$ are roots of $Q_n(x)$. Note that roots of $q(x)$ are $e^{\pm i\theta^*}$. By \eqref{E5} and $G(\theta^*)=0$, we have $Q_n(e^{i\theta^*})=0$. Since $Q_n(x)\in \R[x]$, by conjugation, we also have $Q_n(e^{-i\theta^*})=0$. Hence, $q(x)|Q_n(x)$ in $\R[x]$. It also follows that  $\deg B_n(x)=12n+6=2m$ and $B_n(0)=1$.
\end{proof}
Next, we write out the coefficients of $B_n(x)$ explicitly. Recall that for a set $S$, the function $[x\in S]$ is defined as 
\[[x\in S]=\begin{cases}
    0\text{ if } x\not\in S,\\
    1 \text{ if } x\in S.
\end{cases}\]
\begin{lem}\label{L3}
Write $B_n(x)=\sum\limits_{k\geq 0}b_kx^k$. Then $b_k=0$ for $k>2m$, and for $0\leq k\leq 2m$,
\begin{equation}\label{E9}
b_k=[k\in E]-\dfrac{t}{\sin\theta^*}S_k(\theta^*),
\end{equation}
where $S_k(\theta):=\sum\limits_{e\in E,e<k}\sin\left((k-e)\theta\right)$.
\end{lem}
\begin{proof} Since $q(0)=1$, expand $\dfrac{1-x+x^2}{q(x)}$ as a Maclaurin series.
 \[\dfrac{1-x+x^2}{q(x)}=\sum\limits_{j\geq 0}W_jx^j\]
 Then \[\left(1-2\cos(\theta^*)x+x^2\right)\left(\sum\limits_{j\geq 0}W_jx^j\right)=1-x+x^2.\]
Equating the coefficients of $x^j$ ($j\geq 0$) both sides gives \[W_0=1,\,W_1-2\cos(\theta^*)W_0=-1,\,W_2-2\cos(\theta^*)W_1+W_0=1,\]
and \[W_j-2\cos(\theta^*)W_{j-1}+W_{j-2}=0\quad\forall j\geq 3.\]
We prove by induction that \begin{equation}\label{Che}W_j=-\dfrac{t\sin(j\theta^*)}{\sin\theta^*}\quad\forall j\geq 0.\end{equation}
Since $W_0=1$, \[W_1=2\cos(\theta^*)W_0-1=2\cos\left(\theta^*\right)-1=-t=-\dfrac{t\sin\theta^*}{\sin\theta^*},\]
\[W_2=2\cos(\theta^*)W_1-W_0+1=-2t\cos(\theta^*)=-\dfrac{t\sin(2\theta^*)}{\sin\theta^*}\]
The formula \eqref{Che} is true for $j=0,1,2$. Assume that \eqref{Che} is true for $0,1,2,\ldots,j$, with $j\geq 2$. Then \[\begin{split}W_{j+1}&=2\cos(\theta^*)W_{j}-W_{j-1}\\
&=-\dfrac{t}{\sin\theta^*}\left(2\cos\theta^*\sin j\theta-\sin(j-1)\theta\right)\\
&=-\dfrac{t\sin((j+1)\theta)}{\sin\theta^*}.
\end{split}\]
Now, \[\begin{split}B_n(x)&=Q_n(x)\left(\sum\limits_{j\geq 0}W_jx^j\right)\\
&=\left(\sum\limits_{e\in E}x^e\right)\left(\sum\limits_{j\geq 0}W_jx^j\right).
\end{split}\]
Combining with \eqref{Che}, we obtain that \[\begin{split}b_k&=\sum\limits_{e\in E,e\leq k}W_{k-e}\\
&=[k\in E]-\dfrac{t}{\sin\theta^*}S_k(\theta^*).\end{split}\]
\end{proof}
\begin{lem}\label{Pro1}
Let $c(x)=\sum\limits_{j\geq 0}c_jx^j$ be a 0--1 polynomial with $c_1=0$. Assume that $c(x)=a(x)b(x)$ with $a(x),b(x)\in \R[x]$ and $a(0)=b(0)=1$. Write $a(x)=\sum\limits_{j\geq 0}a_jx^j$ and $b(x)=\sum\limits_{j\geq 0}b_jx^j$. Then $b_1=-a_1$. In particular, if $a_1>0$ then $b_1<0$, and every $\epsilon$ with all coefficients of $a,b$ in $[-\epsilon,1+\epsilon]$ satisfies $\epsilon\geq a_1$.
\end{lem}
\begin{proof} This is obvious since comparing the coefficients of $x$ in $c(x)=a(x)b(x)$ gives \[0=c_1=a_0b_1+a_1b_0=a_1+b_1.\]
\end{proof}
\begin{lem}\label{L5} With the notation $S_k(\theta):=\sum\limits_{e\in E,e<k}\sin\left((k-e)\theta\right)$ as in Lemma \ref{L3}. Then \[|S_k(\dfrac{\pi}{3})|\leq \dfrac{\sqrt{3}}{2}\quad\forall 0\leq k\leq 2m.\]
\end{lem}

\begin{proof}
Write $\sin(j\pi/3)=\dfrac{\sqrt{3}}{2}g(j\pmod{6})$, where \[g(0)=0,\,g(1)=1,\,g(2)=1,\,g(3)=0,g(4)=-1,g(5)=-1.\]
Let $j:=k\pmod{6}$. Since $m\equiv 3$ (mod 6), we have $m-1\equiv 2$ (mod $6$), $m+1\equiv 4$ (mod $6$), $2m\equiv 0$ (mod $6$). To show $|S_k(\dfrac{\pi}{3})|\leq \dfrac{\sqrt{3}}{2}$, it suffices to show that $|T_k|\leq 1$, where  \[T_k:=\sum_{e\in E,e<k}g(k-e).\]
If $k=0$, the sum is empty. If $1\leq k\leq m-1$, the only $e\in E$ such that $e<k$ is $e=0$, so \[|T_k|=|g(j)|\leq 1.\]
If $k=m$, then $e\in \{0,m-1\}$. Hence, \[T_k=g(m)+g(1)=g(3)+g(1)=1.\]
If $k=m+1$, then $e\in \{0,m-1,m\}$. Hence, \[T_k=g(m+1)+g(2)+g(1)=g(4)+1+1=1.\]
Finally, if $m+2\leq k\leq 2m$, then $e$ ranges over $\{0,m-1,m,m+1\}$ (the element $2m$ is excluded since $2m<k$ fails). Hence, \[T_k=g(j)+g(j-2)+g(j-3)+g(j-4)\]
with the arguments read modulo 6. Evaluating for $j=0,1,2,\ldots,5$ gives the value of $T_k$ respectively \[0,\,-1,\,-1,\,0,\,1,\,1.\]
Therefore, in all cases, \[|T_k|\leq 1.\]
\end{proof}
\begin{lem}\label{L6}
With $\theta^*$ as in Lemma \ref{L1} and $S_k$ as in Lemma \ref{L3}, then \[|S_k(\theta^*)|\leq \dfrac{\sqrt{3}}{2}+\dfrac{20}{m}\quad\forall 0\leq k\leq 2m.\]
\end{lem}
\begin{proof}
Fix $k$ with $0\leq k\leq 2m$. For $e\in E$ with $e<k$. Let $u=(k-e)\pi/3$ and $v=(k-e)\theta^*$. Since $0<k-e\leq 2m$ and $\delta^*<1.78m^{-2}$ by Lemma \ref{L1}, we have \[0<v<\dfrac{3.56}{m}<\dfrac{4}{m}.\]
Therefore, \[\begin{split}|\sin(u+v)-\sin u|&=|\sin u(\cos v-1)+\cos u \sin v|\\
&\leq 1-\cos v+\sin v\\
&\leq \dfrac{v^2}{2}+v\\
&<\dfrac{8}{m^2}+\dfrac{4}{m}.\end{split}\]
For $0\leq k\leq 2m$, the sum $S_k$ has at most four terms: the element $2m$ of $E$ never satisfies $2m<k$. Therefore,
\[\left|S_k(\theta^*)-S_k(\dfrac{\pi}{3})\right|\leq \dfrac{32}{m^2}+\dfrac{16}{m}\leq \dfrac{20}{m}.\]
Now applying Lemma \ref{L5} gives \[|S_k(\theta^*)|\leq \dfrac{\sqrt{3}}{2}+\dfrac{20}{m}.\]
\end{proof}
We are now ready to prove Theorem \ref{T1}. Let $t=t_n,\theta^*,$ and let $A_n(x)$, $B_n(x)$, and $C_n(x)$ as in Lemma  \ref{L2}. The bounds of $t_n$ in \eqref{E8} are proved in Lemma \ref{L1}, $A_n|C_n$ with $\deg B_n=12n+6$, and $B_n(0)=1$ are proved in Lemma \ref{L2}. Taking $k=1$ in \eqref{E9}, the only $e\in E$ with $e<1$ is $e=0$, so that $b_1=-t_n$. It remains to prove \eqref{E1} and \eqref{alpha_n}.

$\bullet$ {\bf Upper bound for $\alpha_n$:}
All coefficients of $A_n$ are $1,t_n,t_n,1$, all in $[0,1]$. Consider $B_n$. Then \eqref{E9} gives \[b_k-[k\in E]=-\dfrac{tS_k(\theta^*)}{\sin \theta^*}.\]
Using Lemma \ref{L6} with the notice that $\sin \theta^*\geq \dfrac{\sqrt{3}}{2}$ (this is because $\theta^*\in (\pi/3,\pi/2)$), we have 

\[\begin{split}\max\limits_{0\leq k\leq 2m}\left|\dfrac{tS_k(\theta^*)}{\sin \theta^*}\right|&\leq \dfrac{t}{\sin \theta^*}\left(\dfrac{\sqrt{3}}{2}+\dfrac{20}{m}\right)\\
&=t\cdot \dfrac{\dfrac{\sqrt{3}}{2}}{\sin\theta^*}+\dfrac{t}{\sin\theta^*}\cdot\dfrac{20}{m}\\
&<t\left(1+\dfrac{2}{\sqrt{3}}\cdot \dfrac{20}{m}\right)\\
&<t\left(1+\dfrac{24}{m}\right).\end{split}\]
Since $[k\in E]\in \{0,1\}$, it follows that all coefficients of $B_n$ lie in $[-\alpha,1+\alpha]$ for all $\alpha\geq t\left(1+\dfrac{24}{m}\right)$. 

$\bullet${\bf Lower bound for $\alpha_n$:} The inequality $\alpha_n\geq t_n$ follows immediately from Lemma \ref{Pro1}.

$\bullet$ {\bf Asymptotics:} Combining \[t_n\leq \alpha_n\leq t_n\left(1+\dfrac{24}{m}\right)\]
with \[|t_n-\dfrac{3}{m^2}|\leq \dfrac{7}{m^4},\]
we obtain that 
\[\begin{split}\alpha_n&=\left(\dfrac{3}{m^2}+O\left(\dfrac{1}{m^4}\right)\right)\left(1+O\left(\dfrac{1}{m}\right)\right)\\
&=\dfrac{3}{m^2}+O\left(\dfrac{1}{m^3}\right)\end{split}.\]

\section{Proof of Theorem \ref{T2}}
Given $\epsilon>0$, choose $n$ such that $12m^{-2}<\epsilon$, where $m=6n+3\geq 9$. Take $c(x)=C_n(x)$, $a(x)=A_n(x)$, and $b(x)=B_n(x)$. By Theorem \ref{T1} and \eqref{E8},\[
\alpha_n\leq \left(1+\dfrac{24}{m}\right)t_n\leq \left(1+\dfrac{24}{9}\right)\cdot \dfrac{3.1}{m^2}<\dfrac{12}{m^2}.
\]
Therefore, $\alpha_n<\epsilon$. Now $c(x)=a(x)b(x)$ is a factorization in monic real polynomials with $a(0)=b(0)=1$, all coefficients in $[-\epsilon,1+\epsilon]$, and the coefficient of $x$ in $b(x)$ is $-t_n\not \in \{0,1\}$.
It follows from \eqref{alpha_n} that \[\lim\limits_{n\rightarrow \infty}{(6n+3)^2\alpha_n}=3.\]
Taking logarithms gives  
\[\lim\limits_{n\rightarrow\infty}\dfrac{\log \alpha_n}{\log (6n+3)}=-2.\]
It follows that \[\lim\limits_{n\rightarrow\infty}\dfrac{\log \alpha_n}{\log (12n+9)}=-2.\]
\begin{rem}The polynomial $C_n(x)$ in Theorem \ref{T1} differs from $C_n(x)$ used by Hare \cite{Hare} by only a single term $x^{6n+7}$.
\end{rem}
\begin{rem} Another proof of Lemma \ref{L3} is to use the generating function of Chebyshev polynomials of the second kind (see \cite{Ri}):
\[\sum\limits_{j\geq 0}U_j(u)x^j=\dfrac{1}{1-2ux+x^2}.\]
\end{rem}
\section{Acknowledgment} Part of this work was finished during the 4th International Mathematics Summer Camp (IMSC26) co-sponsored by Beijing Institute
of Mathematical Sciences and Applications (BIMSA) and Yau Mathematical Sciences Center
(YMSC), Tsinghua University. Nguyen Xuan Tho would like to thank the organizers of the camp for their support and hospitality. Pham Lan Huong would like to thanks Vietnam Institute for Advanced Study in Mathematics (VIASM) for their support and funding from 15-8-2026 to 15-9-2026 and from 15-10-2026 to 15-11-2026.

\end{document}